\documentclass[11pt,a4paper]{article}
\title{\bf Littlewood--Paley--Stein Square Functions for the Fractional Discrete Laplacian on $\mathbb{Z}$}
\author{Huaiqian Li\footnote{Email: huaiqian.li@tju.edu.cn \quad Partially supported by the National Key R\&D Program of China (Grant No. 2022YFA1006000) and the National Natural Science Foundation of China (Grant No. 11831014).
}
\quad Liying Mu\footnote{Email: mly18822136120@163.com}
  \vspace{2mm}
\\
{\footnotesize Center for Applied Mathematics, Tianjin University, Tianjin 300072, P. R. China}
}

\date{}

\usepackage{amssymb,amsmath,amsfonts,amsthm,color,mathrsfs}
\usepackage{bbm}
\usepackage[marginal]{footmisc}
\usepackage[]{hyperref}%%%³¬Á´½Ó
\usepackage{color} %%%
\def\R{\mathbb{R}}
\def\E{\mathbb{E}}

\def\D{\mathrm{D}}
\def\Z{\mathbb{Z}}
\def\i{\mathrm{i}}

\def\d{\textup{d}}
\def\D{\textup{D}}
\def\U{\mathrm{U}}

\def\<{\langle}
\def\>{\rangle}
\def\Proof.{\noindent{\bf Proof. }}

\def\newdot{{\kern.8pt\cdot\kern.8pt}}

\newtheorem{theorem}{Theorem}[section]
\newtheorem{lemma}[theorem]{Lemma}
\newtheorem{corollary}[theorem]{Corollary}
\newtheorem{proposition}[theorem]{Proposition}

\theoremstyle{definition}\newtheorem{remark}[theorem]{Remark}

\begin{document}
\allowdisplaybreaks
\maketitle
\makeatletter % '@' is now a normal "letter" for TeX
\renewcommand\theequation{\thesection.\arabic{equation}}
\@addtoreset{equation}{section}
\makeatother % '@' is restored as a "non-letter" character for TeX

\begin{abstract}
We investigate the boundedness of  ``vertical'' Littlewood--Paley--Stein square functions for the nonlocal fractional discrete Laplacian on the lattice $\mathbb{Z}$, where the underlying graphs are not locally finite. When $q\in[2,\infty)$, we prove the $l^q$ boundedness of the square function by exploring the corresponding Markov jump process and applying the martingale inequality. When $q\in (1,2]$, we consider a modified version of the square function and prove its $l^q$ boundedness through a careful in on the generalized carr\'{e} du champ operator. A counterexample is constructed to show that it is necessary to consider the modified version. Moreover, we extend the study to a class of nonlocal Schr\"{o}dinger operators for $q\in (1,2]$.
\end{abstract}

{\bf MSC 2020:} primary 60G51, 42B25; secondary 60J60, 60J74

{\bf Keywords:} Fractional discrete Laplacian; square function; Schr\"{o}dinger operator;  jump process

\section{Introduction}\hskip\parindent
In classical harmonic analysis, the Littlewood--Paley square function plays an important role in the study of the boundedness of Riesz transform, the boundedness of Fourier multipliers, the convergence of non-tangential maximal functions, and so on.  The pivotal status has led to the in-depth investigation of Littlewood--Paley square functions in different settings and for diverse objectives. However, when it comes to applications on partial differential equations for instance, what kind of square function we employ depends on the semigroup chosen. Consequently, it stimulates the intensive study of specific square functions associated with a given semigroup. See e.g. \cite{Coulhon-2003,Lohou1987,Stein-1970,Stein-1970-Topics,St1958}. Last but not the least, the probabilistic counterpart of square functions is the (predictable) quadratic variation of martingales, which have been also investigated a lot; refer to \cite{Sh2002,ShYo92,Ba86,Benn85,Mey1981,Varo1980,Mey1976} for instance.

The motivation of the present work is two fold. On the one hand, the boundedness of ``horizontal'' (i.e., derivative w.r.t. the time variable) Littlewood--Paley--Stein square function for discrete Laplacian was proved in some weight $l^p$ space over the one-dimensional lattice $\Z$ in \cite{ciaurri-2017}, and the regularity and extension problems of the fractional discrete Laplacian on $\Z$  was studied recently in \cite{ciaurri-2018}. On the other hand, the Littlewood--Paley--Stein square function for pure jump L\'{e}vy process on $\R^d$ was investigated in \cite{Ba-2016}, which was also extended recently to the more general setting of Dirichlet forms of pure jump type on metric measure spaces in \cite{LiWang16} and in the setting of nonlocal Schr\"{o}dinger type operators in \cite{Li-2020}; see also the very recent papers \cite{Li2021,LZ2022+} in the Dunkl setting where the corresponding jump process may not be a L\'{e}vy process.

So, it would be interesting to consider the ``vertical'' (i.e., derivative w.r.t. the space variable)  Littlewood--Paley--Stein square function for the fractional discrete Laplacian on lattices. In contrast to related results appeared in the literature, such as Dungey's consideration on uniformly locally finite graphs (see \cite{Dungey-2008}), the graphs underlying the Markov chain generated by the fractional discrete Laplacian on $\Z$ are locally infinite (see Remark \ref{rem-graph} for details).

\section{Preparations and main results}\hskip\parindent
In this section, we aim to present the main results. For this purpose, we shall introduce some notions and notations which will be frequently used below.

Let $\Z$ be the one-dimensional lattice endowed with the counting measure $\#$. Let $q\in[1,\infty]$. For convenience, we use $l^q$ to denote the standard space $l^q(\mathbb{Z},\#)$ with the norm
$$\Vert f\Vert_{q}:=\left(\sum_{x\in \mathbb{Z}}  |f(x)|^q\right)^\frac{1}{q},\quad 1\leq q<\infty,$$
and
$$\Vert f\Vert_{\infty}:=\sup_{x\in \mathbb{Z}}|f(x)|.$$
Note that for any $p\geq q\geq1$, we have $l^q\subset l^p\subset l^\infty$.

Recall that the discrete Laplacian on $\mathbb{Z}$, denoted by $\Delta$, is defined as
$$\Delta u (j)=u (j+1)-2 u (j)+u (j-1),\quad j\in\mathbb{Z},$$
for every function  $u:\mathbb{Z}\rightarrow \mathbb{R}$.

There are a couple of ways to introduce the fractional discrete Laplacian; see e.g. \textcolor{red}{\cite{Bendikov-2016,Cygan, Garo,SSV2012,Stinga-2,Sato2013,Stein-1970-Topics}}. Here we choose the semigroup approach; see e.g. \textcolor{red}{\cite{ciaurri-2018, Sato2013,Stein-1970-Topics}}. We always assume that $0 < s < 1$. For any suitable function $u:\mathbb{Z}\rightarrow \mathbb{R}$, define the fractional discrete Laplacian $L:=(-\Delta)^s$ on $\Z$ as
$$L u =\frac{1}{\Gamma(-s)} \int_{0}^{\infty}(e^{t \Delta} u-u)\, \frac{d t}{t^{1+s}},$$
where $\Gamma$ denotes the Gamma function and $s\Gamma(-s)=-\Gamma(1-s)$. Indeed, we have the following pointwise nonlocal formula, established recently in \cite[Theorem 1.1]{ciaurri-2018}, i.e., for every $u \in \mathcal{D}_s$,
\begin{equation}\label{point-formula}
L u(j)=\sum_{m \in \mathbb{Z}}\big(u(j)-u(m)\big) K_{s} (j-m),\quad j\in\Z,
\end{equation}
where
$$\mathcal{D}_s:=\left\{u:\mathbb{Z} \rightarrow \mathbb{R}:\ \sum_{m \in \mathbb{Z}} |u(m)|{(1+|m|)^{-(1 + 2 s)}}<\infty\right\},$$
and the kernel
$$ K_{s} (m):=
	\frac{4^{s} \Gamma(1 / 2+s)}{\sqrt{\pi}|\Gamma(-s)|} \cdot \frac{\Gamma(|m|-s)}{ \Gamma(|m|+1+s)},\quad m\in \mathbb{Z},$$
with the convention that $K_{s}(0):=0$; moreover, there exists a constant $c_s\in[1,\infty)$, depending on $s$, such that
\begin{equation}\label{kernel-bd}
\frac{c_{s}^{-1}}{|m|^{1+2 s}} \leq K_{s} (m) \leq \frac{c_{s}}{|m|^{1+2 s}},\quad m\in\Z\setminus \{0\}.
\end{equation}
We remark that $l^1\subset\mathcal{D}_s$ for every $s\in(0,1)$.

In the present situation, the operator $-L$ generates a Markov chain with state space $\Z$ and $L$ can be also written in the following way: for any $u\in\mathcal{D}_s$,
$$Lu(i)=\sum_{j\in\Z}\|K_s\|_{l^1}\big(\mathbbm{1}_{i,j}-p_{i,j}\big)u(j),\quad i\in\Z,$$
where $\mathbbm{1}_{i,j}=1$ if $i=j$ and $\mathbbm{1}_{i,j}=0$ otherwise, and
$$p_{i,j}:=\frac{1}{\|K_s\|_{l^1}}K_s(i-j),\quad i,j\in\Z.$$
Note that the one-step transition probabilities of the Markov chain is given by the infinite matrix $(p_{i,j})_{i,j\in\Z}$ such that, for each $(i,j)\in\Z^2$,  $p_{i,j}$ is the probability of jumping from the point $i$ to the point $j$ in the next step. See also \cite[Remark 1.2]{ciaurri-2018}.

As we have mentioned before that the graphs underlying the Markov chain generated by $-L$ with the state space $\Z$ are not locally finite.
\begin{remark}\label{rem-graph}
Let $s\in(0,1)$ and let the triple $(V,E,w)$ be the undirected weighted graph such that $V=\Z$ is the set of vertices and $E$ is the set of edges with edge weight $w_{i,j}:=K_{s} (i-j)$ for every $i,j\in\Z$. Assume that the weight on the vertices is constant. Then the graph Laplacian on $(V,E,w)$ is expressed as \eqref{point-formula}. Note that the Markov chain $(p_{i,j})_{i,j\in\Z}$ above has arbitrarily long jumps. Indeed, for any $i,j\in\Z$ with $i\neq j$, the probability to jump from $i$ to $j$ is $p_{i,j}$ which is comparable to $|i-j|^{-(1+2s)}$ by \eqref{kernel-bd}. Hence, the graph $(V,E,w)$ is locally infinite (equivalently, $K_s$ has unbounded support), which is in contrast to Dungey's setting in \cite{Dungey-2008}. For some basics on the graph theory, see e.g. \cite{Grig2018}.
\end{remark}

In the discrete space, it seems that the most natural gradient operator is the difference operator $ \D$, defined as
$$\D f (x)=f(x+1)-f(x), \quad x\in\Z,$$
for every function $f:\Z\rightarrow\R$. However, we also introduce the the modulus of gradient $|\nabla\cdot|$ and the modulus of the modified gradient $|\widetilde{\nabla}\cdot |$ as candidates for the ``gradient'' operator. For every $f\in \mathcal{D}_s$ and every $ x\in \mathbb{Z}$, define
\begin{equation*}\begin{split}
|\nabla f| (x)&=\left(\sum_{y\in\Z}K_s(y-x)[f(x)-f(y)]^{2}\right)^{\frac{1}{2}}, \\
|\widetilde{\nabla} f| (x)&=\left(\sum_{\{y\in \mathbb{Z}:\ |f(x)|>|f(y)|\}} K_s(y-x) [f(x)-f(y)]^2 \right)^\frac{1}{2}.
\end{split}\end{equation*}
In the present work, we are interested in Littlewood--Paley--Stein square functions associated with $|\D\cdot|$, $|\nabla\cdot|$ and  $|\widetilde{\nabla}\cdot |$, respectively.

Let $s\in(0,1)$ and let $(e^{-tL})_{t\geq0}$ be the semigroup corresponding to $L$. The ``vertical'' Littlewood--Paley--Stein square functions that we consider are defined as follows: for any $f\in l^1$ and any $x\in\Z$,
\begin{equation*}\begin{split}
	G(f)(x)&:=\left(\int_{0}^{\infty} | \nabla e^{-tL} f| ^{2}(x)\, \d t\right)^{\frac{1}{2}},\\
\widetilde{G}(f)(x)&:=\left(\int_{0}^{\infty} | \widetilde{\nabla} e^{-tL} f| ^{2}(x)\, \d t\right)^{\frac{1}{2}},\\
H(f)(x)&:=\left(\int_{0}^{\infty} | \D e^{-tL} f| ^{2}(x)\, \d t\right)^{\frac{1}{2}}.
\end{split}\end{equation*}

 Let $q\in[1,\infty]$. Our aim is to establish the boundedness of the above square functions in $l^q$. We say that $G$ is bounded in $l^q$ if $G$ extends to $l^q$ and there exists a constant $c_q>0$, depending only on $q$, such that
 $$\|G(f)\|_{q}\leq c_q\|f\|_{q},\quad f\in l^q.$$
The same for $\widetilde{G}$ and $H$. The main result is presented in the following theorem.
\begin{theorem}\label{main}
For every $q\in(1,2]$, the square functions $\widetilde{G}$ and $H$ are bounded in $l^q$. For every $q\in[2,\infty)$, the square functions $G$, $\widetilde{G}$ and $H$ are bounded in $l^q$.
\end{theorem}

Some remarks on Theorem \ref{main} are in order.
\begin{remark} (1) For $q\in(1,2)$, the square function $G$ may not be bounded; see the counterexample given in Section 3. So, in this case,  considering the modified version $\widetilde{G}$ is necessary.

(2) For $q\in(1,2]$, in order to highlight the flexibility of our method, we may consider the Littlewood--Paley--Stein square function associated with a class of nonlocal Schr\"{o}dinger operators and prove the boundedness in $l^q$ by adapting the approach employed for the proof of Theorem \ref{main}; see Corollary \ref{schro}.

(3) Consider the discrete Laplacian $\Delta$ and its semigroup $(e^{t\Delta})_{t\geq0}$. Define the ``horizontal'' Littlewood--Paley--Stein square function or $g$-function as
 $$g(f)=\left(\int_0^\infty t\Big|\frac{\partial}{\partial t}e^{t\Delta}f\Big|^2\,\d t\right)^{\frac{1}{2}}.$$
It is well known that, by the general result \cite[Corollary 1 on page 120]{Stein-1970-Topics} for Markov semigroups on metric measure spaces,  $g$ is bounded in $l^q$ for all $1<q<\infty$. In the recent paper \cite{ciaurri-2017}, it is proved that $g$ is bounded in the weighted $l^q$ space for all $q\in(1,\infty)$, where the weight is of Muckenhoupt type.
\end{remark}

Let $\U$ be a non-negative function defined on $\Z$ and $s\in(0,1)$. Consider the nonlocal Schr\"{o}dinger operator on $\Z$, i.e.,
$$L_{\U}:=(-\Delta)^s+\U,$$
which is understood in the sense of quadratic forms. Assume that $L_\U$ generates a strongly continuous, contractive, symmetric and sub-Markovian semigroup on $l^2$, denoted by $(e^{-tL_\U})_{t\geq0}$. Note that $(e^{-tL_\U})_{t\geq0}$ is the so-called ``symmetric diffusion semigroup''  without the conservation property in the sense of the second paragraph on page 65 in \cite{Stein-1970-Topics}. As a consequence, $(e^{-tL_\U})_{t\geq0}$ may be extended to a strongly continuous, positive and contractive semigroup on $l^q$ for all $q\in[1,\infty)$. For simplicity, we keep using the same notation. Then, for each $q\in [1,\infty)$, we have the semigroup domination, i.e.,
$$0\leq e^{-tL_\U}f\leq e^{-tL}f,\quad t\geq0,\,f\in l^q_+,$$
where $l^q_+$ stands for the cone of nonnegative functions in $l^q$. We refer to
\cite{Demuth-2000,Davies89} for a comprehensive study on Schr\"{o}dinger operators. 

For a suitable function $f:\Z\rightarrow\R$, define the Littlewood--Paley--Stein square function in this case as
$$\widetilde{G}_{\U}(f)=\left(\int_{0}^{\infty}\big( | \widetilde{\nabla}e^{-tL_\U}f|^{2}+|\sqrt{\U}e^{-tL_\U}f|^2\big)\,\d t\right)^{\frac{1}{2}},$$
Then we have the following result.
\begin{corollary}\label{schro}
For every $q\in(1,2]$, the square functions $\widetilde{G}_{\U}$ is bounded in $l^q$, i.e., there exists a constant $c_q>0$, depending only on $q$, such that
 $$\|\widetilde{G}_{\U}(f)\|_{q}\leq c_q\|f\|_{q},\quad f\in l^q.$$
\end{corollary}

In Section 3, we construct a concrete counterexample to show that $G$ is not always bounded in $l^q$ for all $q\in(1,2)$. In Section 4, we present the proofs for Theorem \ref{main} when $q\in(1,2]$ and Corollary \ref{schro}. In Section 5, we prove Theorem \ref{main} when $q\in[2,\infty)$.

In what follows, for convenience, we also use $P_t$ (resp. $P_t^\U$) to denote $e^{-tL}$ (resp. $e^{-tL_\U}$) for every $t\geq0$.

\section{A counterexample}\hskip\parindent
Let $(h_t)_{t > 0}$ be the heat kernel corresponding to $\Delta$. It is known that
$$h_{t}(x,y)=\frac{e^{-2t}}{\pi}\int_{0}^{\pi} e^{2t\cos(u)}\cos(|x-y|u)\, \d u,\quad x,y\in\Z,\,t>0;$$
see e.g. \cite{ciaurri-2017}.

Let $s\in(0,1)$ and $t>0$. We use $\i$ to denote the imaginary unit. Set
\begin{eqnarray*}
f_{s, t}(\lambda):=
\begin{cases}
\begin{aligned}
\frac{1}{2 \pi \i} \int_{a-\i \infty}^{a+\i \infty} e^{z \lambda-t z^{s} }\,\d z, \quad&{\lambda \geq 0,\,a>0}, \\
0, \quad&{\lambda<0}.\end{aligned}
\end{cases}
\end{eqnarray*}
Then $f_{s, t}(\lambda)\geq0$ for every $\lambda\geq0$. Let $q\in[1,\infty]$. For every $u\in l^q$, we have
\begin{eqnarray*}
P_tu(x) &=&\int_0^{\infty} f_{s, t}(\lambda) e^{\lambda\Delta}u(x)\, \d \lambda\\
&=&\int_0^{\infty} f_{s, t}(\lambda) \sum_{m\in\Z}h_t(x,m)u(m)\, \d \lambda,\quad x\in\Z.
\end{eqnarray*}
In addition, we have the formula
\begin{equation}\label{minus-exp-fomula}
e^{-ta^s}=\int_{0}^{\infty}e^{-\lambda a}f_{s,t}(\lambda)\, \d\lambda, \quad t>0,\,a>0.
\end{equation}
For these facts, refer to \cite[Section 9, Chapter IX]{Yosida-1980} and \cite{SSV2012} for instance.

Now take
\begin{eqnarray*}
f(x)=
\begin{cases}
	0,\quad&{x\in\mathbb{Z}\setminus \{1\},} \\
	1,\quad&{ x=1.}
\end{cases}
\end{eqnarray*}
It is clear that $f\in l^q$ for all $q\in[1,\infty]$, and
\begin{eqnarray}\label{Pts}
P_tf(x)=\int_{0}^{\infty} f_{s, t}(\lambda)\frac{e^{-2\lambda}}{\pi}\int_{0}^{\pi} e^{2\lambda \cos(u)}\cos(|x-1|u)\,  \d u  \d\lambda,\quad x\in\Z.
\end{eqnarray}

Assume that $x\in\Z$ and $x>1$. Then, by \eqref{Pts} and the first inequality in \eqref{kernel-bd}, we have
\begin{eqnarray*}
 &&G(f)^{2}(x)\\
 &=&\int_{0}^{\infty} \sum_{y\in \mathbb{Z}\setminus \{0\}}[P_{t} f(x+y)-P_{t} f(x)]^{2}K_s(y) \,\d t\\
 &\geq& \int_{0}^{\infty} c_s^{-1}\sum_{y\in \mathbb{Z}\setminus \{0\}}\frac{[P_{t} f(x+y)-P_{t} f(x)]^{2}}{|y|^{1+2s}} \,\d t\\
 &=&\int_{0}^{\infty} \sum_{y\in \mathbb{Z}\setminus\{0\}} \frac{c_s^{-1}}{|y|^{1+2s}}\left(\int_{0}^{\infty} f_{s, t}(\lambda)\frac{e^{-2\lambda}}{\pi}\int_{0}^{\pi} e^{2\lambda \cos(u)}[\cos(|x+y-1|u)-\cos(|x-1|u)]\,  \d u  \d\lambda \right) ^{2}\,\d t\\
 &\geq& \int_{0}^{\infty}\frac{c_s^{-1}}{|1-x|^{1+2s}}\left(\int_{0}^{\infty} f_{s, t}(\lambda)\frac{e^{-2\lambda}}{\pi}\int_{0}^{\pi} e^{2\lambda \cos(u)}[1-\cos(|x-1|u)]\,  \d u   \d \lambda\right) ^{2}\,\d t\\
&\ge&\int_{0}^{\infty}\frac{c_s^{-1}}{|1-x|^{1+2s}} \left(\int_{0}^{\infty} f_{s, t}(\lambda)e^{-4\lambda} \d\lambda \right) ^{2}\,\d t \\
&=&\frac{c_s^{-1}}{|1-x|^{1+2s}}\int_{0}^{\infty} \left( e^{-t4^{s}} \right)^{2}\, \d t,
\end{eqnarray*}
where we employed \eqref{minus-exp-fomula} in the last equality.

In particular, take $s=\frac{1}{4}$. Then, $4\int_{0}^{\infty} e^{-2\sqrt{2}t} \,\d t=\sqrt{2}$, and we immediately have
$$|G(f)(x)|\ge \frac{\sqrt[4]{2}c_{1/4}^{-1}}{2|1-x|^{3/4}},\quad x\in\Z,\, x>1.$$
Thus
$$\begin{aligned}
	\Vert G(f) \Vert^{q}_{q}&=\sum_{x\in \mathbb{Z}}|G(f)(x)|^{q}\\
	&\ge \sum_{m=2}^{\infty}|G(f)(m)|^{q}\\
	&\ge{\left(\frac{\sqrt[4]{2}c_{1/4}^{-1}}{2}\right)}^q\sum_{m=1}^\infty\frac{1}{m^{3q/4}}.
\end{aligned}$$
It is clear that, for every $q\in(1,\frac{4}{3}]$, we have $\Vert G(f) \Vert_{q}=\infty.$

Let us note in passing that the above example also shows that the ``local doubling'' property is necessary in the paper \cite{Dungey-2008}.

\section{The case $1<q\leq2$}\hskip\parindent
In this section, we aim to prove the boundedness of $\widetilde{G}$ and $H$ in $l_q$ for all $q\in(1,2]$, as well as Corollary \ref{schro}. At first, let us give a brief description of the  idea of proof.

Let $\mathbb{M}$ be a Riemannian manifold, $\Delta_\mathbb{M}$ be the Laplace--Beltrami operator, $\nabla_\mathbb{M}$ be the gradient operator, and $|\cdot|$ be the length in the tangent space. Let $q\in(1,2]$. The classic argument used by Stein (see \cite{Stein-1970,Stein-1970-Topics}) depends on the following chain rule
\begin{equation*}\Delta_\mathbb{M}f^q=q(q-1)f^{q-2}|\nabla_\mathbb{M} f|^{2}-qf^{q-1}\Delta_\mathbb{M }f, \end{equation*}
for every $0<f\in C^\infty(\mathbb{M})$. However, this chain rule is no longer valid for fractional discrete Laplacian $(-\Delta)^s$. Following the idea of Dungey \cite{Dungey-2008} (in the setting of locally finite graphs),  we introduce the pseudo-gradient operator $\Gamma_q$, i.e.,
$$\Gamma_{q}(f)=q fL f-f^{2-q} Lf^{q},$$
for suitable nonnegative functions $f$ defined on $\Z$, and consider
$$\mathcal{H}_{q}(f)=\left(\int_{0}^{\infty}  \Gamma_{q}\left(e^{-t L} f\right)\,\d t \right)^{1 / 2},\quad 0\leq f\in l^1.$$
By adapting Stein's argument, we may prove the $l^q$ boundedness of $\mathcal{H}_{q}$; see Proposition \ref{LE4.2}. Hence, in order to show the $l^q$ boundedness of $\widetilde{G}$, $G$ and $H$, the problem left is to compare  $\Gamma_{q}$ and $|\widetilde{\nabla}\cdot|$, $|\nabla\cdot|$ and $|\D\cdot|$,  which hinges on a deep understanding of $\Gamma_q$.

Given $f\in l^q$, define the semigroup maximal function $f^*$ by
$$f^*(x):=\sup_{t>0}|e^{-tL} f|(x),\quad x\in \mathbb{Z}.$$
Then we have the following lemma which is proved in \cite[page 73]{Stein-1970-Topics} for much more general contractive, symmetric sub-Markovian semigroups.
\begin{lemma}\label{LE4.1}
For every $q\in (1,\infty]$, there exists a constant $c_q>0$ depending only on $q$ such that
$$\Vert f^*\Vert_{q} \leq c_q \Vert f\Vert_{q},\quad q\in(1,\infty].$$
\end{lemma}

We have the following boundedness property for $\mathcal{H}_q$.
\begin{proposition}\label{LE4.2}
Let $q\in(1,2]$. There exists a constant $c_q>0$ depending only on $q$ such that
$$\Vert \mathcal{H}_{q} (f) \Vert_{q} \leq c_q \Vert f \Vert_{q},\quad 0 \leq f \in l^1.$$
\end{proposition}
\begin{proof} Let  $q\in(1,2]$ and $0 \leq f \in l^1$. Set $u_{t}:=P_t f$ for every $t\geq0$.  %For convenience, we let $L$ stand for $(-\Delta)^s$.
We have
\begin{eqnarray*}
u_{t}^{q-2} \Gamma_{q}(u_{t})&=&q u_{t}^{q-1} L u_{t} -L u_{t}^{q} \\
&=&q u_{t}^{q-1}(\partial_{t}+L) u_{t}-(\partial_{t}+L)(u_{t}^{q}).
\end{eqnarray*}
Since $(\partial_{t}+L) u_{t}=0$, we immediately get
$$\Gamma_{q}(u_{t})=-u_{t}^{2-q}(\partial_{t}+L)(u_{t}^{q}).$$
Then
\begin{eqnarray*}
(\mathcal{H}_{q} f)^{2}&=&\int_{0}^{\infty}  \Gamma_{q}(u_{t})\,\d t\\
&=&-\int_{0}^{\infty}  u_{t}^{2-q}(\partial_{t}+L)(u_{t}^{q})\,\d t\\
&\leq& (f^{*})^{2-q} J,
\end{eqnarray*}
 where
	$$J(x):=-\int_{0}^{\infty} \left(\partial_{t}+L\right)u_{t}^{q}(x)\,\d t\geq0,\quad x\in\Z.$$

According to the H{\"o}lder's inequality and Lemma \ref{LE4.1}, there exists a constant $c_q'>0$ such that
	\begin{eqnarray*}
		\sum_{x\in\Z} \big(\mathcal{H}_{q} f(x)\big)^{q} & \leq &\sum_{x\in\Z}  f^{*}(x)^{q(2-q) / 2} J(x)^{q / 2} \\
		& \leq&\left[\sum_{x\in\Z}  f^{*}(x)^{q}\right]^{(2-q) / 2}\left[\sum_{x\in\Z}  J(x)\right]^{q / 2}\\
      &\leq&c_{q}'\|f\|_{q}^{q(2-q)/2}\left[\sum_{x\in\Z}  J(x)\right]^{q / 2}.
	\end{eqnarray*}

By \eqref{point-formula}, for each  $g \in l^{1}$,
$$\sum_{x\in\Z} (L g)(x)=0.$$
Hence
	\begin{eqnarray*}
		\sum_{x\in\Z}  J(x) &=&-\int_{0}^{\infty} \left(\sum_{x\in\Z}  \partial_{t}u_{t}^{q}(x)\right)\,\d t\\
&=&-\int_{0}^{\infty} \partial_{t}\left(\sum_{x\in\Z}u_{t}^{q}(x)\right)\,\d t \\
		& \leq& \sum_{x\in\Z}  f(x)^{q}=\|f\|_{q}^{q} .
	\end{eqnarray*}

Combining the above estimates together, we finally arrive at
	$$
	\sum_{x\in\Z} \big(\mathcal{H}_{q} f(x)\big)^{q} \leq c_{q}\|f\|_{q}^{q},\quad 0\leq f\in l^1,
	$$
for some constant $c_q>0$ depending only on $q$.
\end{proof}

In order to control the modulus of difference operator $|\D\cdot|$ and the modulus of modified gradient $|\widetilde{\nabla}\cdot|$ by $\Gamma_q(\cdot)$ pointwise, we need the following lemma which provides an explicit expression for $\Gamma_q$.
\begin{lemma}\label{LE4.3}
	Let $q\in (1,2]$. For every $0\leq f \in l^1$ and every $x\in\Z$,
\begin{equation}\begin{split}\label{gamma-q}
	&\Gamma_{q}(f)(x) \\
	&=\sum_{y\in \mathbb{Z}} K_s(x-y)\big[q f(x)(f(x)-f(y))-f(x)^{2-q}\big(f(x)^{q}-f(y)^{q}\big)\big] \\
		&=q(q-1) \sum_{y\in \mathbb{Z}} K_s(x-y)\big(f(x)-f(y)\big)^{2} \int_{0}^{1}  \frac{(1-u) f(x)^{2-q}}{\big[(1-u) f(x)+u f(y)\big]^{2-q}}\, \d u.
\end{split}\end{equation}
\end{lemma}
\begin{proof}
According to the definition of $\Gamma_{q}$, the first equation in \eqref{gamma-q} can be obtained by simple calculation, and we are left to prove the second equation.

Consider the Taylor expansion of the function $t\mapsto t^q$. For any $s,t\geq0$ with $s\neq t$, we have
\begin{eqnarray*}
		t^{q}-s^{q} &=&q s^{q-1}(t-s)+q(q-1) \int_{s}^{t}(t-\tau) \tau^{q-2}\, \d \tau \\
		&=&qs^{q-1}(t-s)+q(q-1)(t-s)^{2} \int_{0}^{1} \frac{1-u}{[(1-u) s+u t]^{2-q}}\, \d u.
	\end{eqnarray*}
Let $0\leq f \in l^1$ such that $f(x)\neq f(y)$. Taking $s=f(x)$ and $t=f(y)$ in the above equality, we have
\begin{eqnarray*}
		&&q f(x)(f(x)-f(y))-f(x)^{2-q}(f(x)^{q}-f(y)^{q}) 	\\
		&=&f(x)^{2-q}[f(y)^{q}-f(x)^{q}-q f(x)^{q-1}(f(y)-f(x))] \\
		&=&f(x)^{2-q} q(q-1)(f(y)-f(x))^{2} \int_{0}^{1} \frac{1-u}{[(1-u) f(x)+u f(y)]^{2-q}}\, \d u \\
		&=&q(q-1)(f(x)-f(y))^{2} \int_{0}^{1} \frac{(1-u) f(x)^{2-q}}{[(1-u) f(x)+u f(y)]^{2-q}}\, \d u.
	\end{eqnarray*}
Thus, by the first equality of \eqref{gamma-q}, we obtain
	$$\Gamma_{q}(f)(x)=q(q-1) \sum_{y\in \mathbb{Z}} K_s(x-y)[f(x)-f(y)]^{2} \int_{0}^{1}
\frac{(1-u) f(x)^{2-q}}{((1-u) f(x)+u f(y))^{2-q}} \,\d u.$$
\end{proof}

The next lemma shows that we can bound $|\D f|$ by $\Gamma_q(f)$ in the pointwise sense.
\begin{lemma}\label{gamma-diff-bd}
For any $q\in(1,2]$,
\begin{eqnarray}\label{gamma-bd}
0\leq |\widetilde{\nabla}f|^2(x)\leq \frac{2}{q(q-1)}\Gamma_q (f)(x), \quad 0\leq f\in l^1,\,x\in \mathbb{Z},
	\end{eqnarray}
and there exists a constant $c_q>0$ such that
\begin{eqnarray}\label{diff-bd}
 |\D f|^2 (x)\leq c_q \left[\Gamma_q(f)(x+1)+\Gamma_q(f)(x)\right], \quad 0\leq f\in l^1,\,x\in \mathbb{Z}.
 	\end{eqnarray}
\end{lemma}
\begin{proof} We divide the proof into two parts. Let $ 0\leq f\in l^1$ and $x\in \mathbb{Z}$.

(i) \underline{Proof of \eqref{gamma-bd}}. For $f(x)\geq f(y)$, we have $(1-u)f(x)+uf(y)\leq f(x),\ u\in[0,1].$ Then
$$ \int_{0}^{1} \frac{(1-u)f(x)^{2-q}}{((1-u)f(x)+uf(y))^{2-q} }\,\d u\geq \int_{0}^{1} \frac{(1-u)f(x)^{2-q}}{f(x)^{2-q}}\, \d u= \frac{1}{2}, $$
and hence, according to Lemma \ref{LE4.3},
\begin{eqnarray*}
\Gamma_{q}(f)(x)&\geq&
%&=q(q-1) \sum_{y\in \mathbb{Z}} p(x, y)(f(x)-f(y))^{2} \int_{0}^{1}  \frac{(1-u) f(x)^{2-q}}{((1-u) f(x)+u f(y))^{2-q}} du,\\
\frac{q(q-1)}{2}  \sum_{\{y\in\Z:\ f(x)\geq f(y)\}}K_s(x-y)(f(x)-f(y))^{2}\\
&=& \frac{q(q-1)}{2} |\widetilde{\nabla} f|^{2}(x).
	\end{eqnarray*}

(ii) \underline{Proof of \eqref{diff-bd}}. If $f(x+1)\geq f(x)$, then $ (1-u)f(x)+uf(x)\leq f(x+1)$ for every $u\in[0,1]$, and hence,
$$\int_{0}^{1} \frac{(1-u)f(x+1)^{2-q}}{[(1-u)f(x+1)+uf(x)]^{2-q}}\, \d u \geq \int_{0}^{1} (1-u)\,\d u =\frac{1}{2}.$$
Thus, by \eqref{kernel-bd}, we obtain
\begin{eqnarray}\label{diff-bd-1}
|\D f|^2(x)&=&[f(x+1)-f(x)]^{2}\cr
 & \leq& 2[f(x+1)-f(x)]^{2} \int_{0}^{1} \frac{(1-u)f(x+1)^{2-q}}{[(1-u)f(x+1)+u f(x)]^{2-q}}\, \d u \cr
& \leq&  c_{q}' \Gamma_q(f)(x+1),
\end{eqnarray}
for some constant $c_q'>0$.

If $f(x)\geq f(x+1)$,  then $ (1-u)f(x+1)+uf(x)\leq f(x)$ for every $u\in[0,1]$, and hence,
$$\int_{0}^{1} \frac{(1-u)f(x)^{2-q}}{[(1-u)f(x)+uf(x+1)]^{2-q}}\,\d u \geq \int_{0}^{1} (1-u)\, \d u =\frac{1}{2}.$$
Thus, by \eqref{kernel-bd} again, we get
\begin{eqnarray}\label{diff-bd-2}
|\D f|^2(x)&=&(f(x+1)-f(x))^{2}\cr
 & \leq& 2(f(x+1)-f(x))^{2} \int_{0}^{1} \frac{(1-u)f(x)^{2-q}}{[(1-u) f(x)+u f(x+1)]^{2-q}}\, \d u \cr
& \leq&  c_{q}'' \Gamma_q(f)(x),
\end{eqnarray}
for some constant $c_q''>0$.

Putting \eqref{diff-bd-1} and \eqref{diff-bd-2} together, we finally have
$$|\D f|^2 (x)\leq c_q \big(\Gamma_q(f)(x+1)+\Gamma_q(f)(x)\big),\quad x\in \mathbb{Z},$$
for some constant $c_q>0$.
\end{proof}

Now we are ready to prove the Littlewood--Paley--Stein estimate for $\widetilde{G}$ and $H$.
\begin{proof}[Proof of Theorem \ref{main} on $\widetilde{G}$ and $H$]
Let $q\in(1,2]$. Since $l^1$ is dense in $l^q$, by standard approximation, it suffices to assume $f\in l^1$.

(1) \underline{\emph{Boundedness for $\widetilde{G}$}}.  Similar as the proof of \cite[Proposition 2.6]{LiWang16}, we may assume that $0\leq f\in l^1$. Applying \eqref{gamma-bd} and Proposition \ref{LE4.2}, we deduce that
\begin{eqnarray*}
		\Vert \widetilde{G} (f)(x)\Vert_{q}^{q}&=&\sum_{x\in \mathbb{Z}}\left(\int_{0}^{\infty}  |\widetilde{\nabla} P_tf(x)|^{2}\, \d t\right)^{q/2}\\
		&\leq& \sum_{x\in \mathbb{Z}} \left(\frac{2}{q(q-1)} \int_{0}^{\infty} \Gamma_{q} (P_{t} f)(y)\,\d t\right)^{q/2}\\
		&\leq& c_1 \sum_{x\in \mathbb{Z}} \big( \mathcal{H}_{q} f(x) \big)^{q}\\
&\leq&c_2\|f\|_q^q, \quad 0\leq f\in l^1,
\end{eqnarray*}
for some constants $c_1,c_2>0$ depending only on $q$, which finish the proof of boundedness for $\widetilde{G}$ in $l^q$.

(2) \underline{\emph{Boundedness for $H$}}. By the sublinear property of $H$, it is enough to prove the case when $0\leq f\in l^1$.
Since $q\in(1,2]$, by \eqref{diff-bd} and Proposition \ref{LE4.2}, we have
\begin{eqnarray*}
\Vert H(f)\Vert_{q}^{q}&=&%\sum_{x\in \mathbb{Z}} |\mathcal{H}(f)|^q,\\
\sum_{x\in \mathbb{Z}} \left(\int_{0}^{\infty} | \D P_tf| ^{2}(x)\, \d t\right)^{\frac{q}{2}}\\
&\leq& c_3 \sum_{x\in \mathbb{Z}} \left(\int_{0}^{\infty} \big[\Gamma_q ( P_tf)(x+1)+\Gamma_q ( P_t f)(x)\big]\, \d t \right) ^{\frac{q}{2}}\\
&\leq& c_3\sum_{x\in \mathbb{Z}} \left[\left(\int_{0}^{\infty} \Gamma_q (P_tf)(x+1)\, \d t \right) ^{\frac{q}{2}}+\left(\int_{0}^{\infty} \Gamma_q (P_tf)(x)\, \d t \right) ^{\frac{q}{2}}\right]\\
&\leq& c_4  \Vert f \Vert_q^q, \quad 0\leq f\in l^1,
\end{eqnarray*}
for some constants $c_3,c_4>0$ depending only on $q$, which completes the proof of boundedness for $H$ in $l^q$.
\end{proof}

Now we turn to consider the boundedness of $\widetilde{G}_{\U}$. The idea of proof is similar to the null potential case tackled above. So we give a brief description on the proof with necessary adaption.

Let $q\in(1,2]$. We introduce the pseudo-gradient associated with $L_\U=(-\Delta)^s+\U$. For any suitable function $f\geq0$, let
$$\Gamma_{q,\U} (f)=qfL_{\U} f-f^{2-q}L_{\U} f^{q}.$$

Let $P_t^\U=e^{-tL_\U}$, $t\geq0$. For any $f\in l^q$, define the semigroup maximal function $f^*_{\U}$ by $$f^*_{\U}=\sup_{t>0}|P_{t}^{\U}f|.$$ Since $(P_{t}^{\U})_{t\geq0}$ is a contractive, symmetric sub-Markovian semigroup, we also have
$$\Vert f^*_{\U}\Vert_q \leq c_q \Vert f\Vert_q,\quad f\in l^q,$$
for some constant $c_q>0$ depending only on $q$.

Consider
$$(\mathcal{H}_{q,\U}f)(x)=\left(\int_{0}^{\infty} \Gamma_{q,\U}(P_{t}^{\U}f)(x)\,\d t\right)^{\frac{1}{2}},\quad x\in\Z,\,0\leq f\in l^1.$$
Then, similar as the proof of Proposition \ref{LE4.2}, there exists a constant $c_{q}>0$, depending only on $q$, such that
\begin{eqnarray}\label{HV-bd}
\Vert \mathcal{H}_{q,\U} f\Vert_{q}\leq c_{q}\Vert f\Vert_{q},\quad 0\leq f\in l^1.
\end{eqnarray}

Now we begin the proof of Corollary \ref{schro}.
\begin{proof}[Proof of Corollary \ref{schro}]
Let $q\in (1,2]$. By standard approximation, it suffices to assume that $f\in l^1$. Due to the lack of sublinear property of $ \widetilde{G}_{\U}$, we may consider $|f|$ instead of $f$ as the proof presented in \cite[Section 3]{Li-2020}. So we assume in addition that $f\geq0$. By the nonnegativity of $\U$ and \eqref{gamma-bd}, we have
\begin{eqnarray*}
		0&\leq& |\widetilde{\nabla} f|^{2}(x)+\U(x)f^{2}(x)\\
&\leq&\frac{2}{q(q-1)}\Gamma_{q}(f)(x)+\U(x) f^2(x)\\
		&\leq&\frac{2}{q(q-1)}\Gamma_{q,\U} (f)(x),\quad x\in \mathbb{Z}.
\end{eqnarray*}
Combining this with \eqref{HV-bd}, we derive that
\begin{eqnarray*}
		\Vert \widetilde{G}_{\U} (f)\Vert_{q}^{q}&=&\sum_{x\in \mathbb{Z}}\left(\int_{0}^{\infty}\big( | \widetilde{\nabla}P_t^\U f|^{2}(x)+|\sqrt{\U}P_t^\U  f|^2(x)\big)\,\d t\right)^{q/2}\\
		&\leq& \sum_{x\in \mathbb{Z}} \left(\frac{2}{q(q-1)} \int_{0}^{\infty} \Gamma_{q,\U} (P_{t}^{\U} f)(x)\,\d t\right)^{q/2}\\
		&=& \left(\frac{2}{q(q-1)}\right)^{q/2} \sum_{x\in \mathbb{Z}} \big( \mathcal{H}_{q,\U} f \big)^{q}(x)\\
&\leq& c_q \Vert f \Vert_q^q,\quad 0\leq f\in l^1,
\end{eqnarray*}
for some constant $c_q>0$ depending only on $q$.
\end{proof}

\section{The case $2\leq q<\infty$}\hskip\parindent
In this section, we turn to prove the Littlewood--Paley--Stein estimate for $G$ in $l^q$ for all $q\in[2,\infty)$. The idea of proof is motivated by \cite{Ba-2016} for $\R^d$-valued L\'{e}vy processes (see also the recent \cite{LZ2022+} where the Markov jump process allows to be not a L\'{e}vy process). Refer to \cite{Sato2013,Applebaum} for more details on L\'{e}vy processes.

Let $(X_t)_{t\geq0}$ be the discrete Markov process generated by $-L$ with state space $\Z$. Denote by $(\mathcal{F}_t)_{t\geq0}$ the natural filtration of the process $(X_t)_{t\geq0}$. Fix $f\in l^1$ and $T>0$. Consider
$$M_{t}:=P_{T-t} f\left(X_{t}\right)-P_{T} f(X_0), \quad 0\leq t\leq T.$$

Then, as in \cite{Ba-2016} and \cite{Li-2020}, we have the following result.  For each $x\in\Z$, let $\E_x$ be the expectation of the process $(X_t)_{t\geq0}$ with initial distribution $\delta_x$, where  $\delta_\cdot$ stands for the Dirac measure.
\begin{lemma}\label{pre-var}
$(M_t, \mathcal{F}_t)_{0\leq t\leq T}$ is a square integrable martingale with $M_0=0$, and
$$\langle M\rangle_{t}=\int_{0}^{t} \sum_{y \in \mathbb{Z}} \big[P_{T-s} f(X_{s}+y)-P_{T-s} f(X_{s})\big]^{2} K_s(y)\, \d s,\quad t\in[0,T],$$
where $\langle M\rangle_{t}$ stands for the predictable quadratic variation of $M_t$.
\end{lemma}
\begin{proof} The first assertion is clear. Let $x\in \mathbb{Z}$ and $0\le t\le T$. Then
\begin{eqnarray*} \mathbb{E}_{x}\left(M_{t}^{2}\right) &=&\mathbb{E}_{x}\big[\big(P_{T-t} f\left(X_{t}\right)-P_{T} f\left(X_{0}\right)\big)^{2}\big] \\ &=&\mathbb{E}_{x}\big[(P_{T-t} f)^{2}(X_{t})\big]-2 P_{T} f(x) \mathbb{E}_{x}\big[ P_{T-t} f(X_{t})\big]+\left(P_{T} f\right)^{2}(x) \\ &=&P_{t}\left(P_{T-t} f\right)^{2}(x)-\left(P_{T} f\right)^{2}(x).
\end{eqnarray*}
 Hence, for any $0\leq s<t\leq T$, we have
\begin{equation*}\label{eq1}
	\begin{split}
\mathbb{E}_{x}(M_{t}^{2}-M_{s}^{2}) &=P_{t}(P_{T-t} f)^{2}(x)-P_{s}(P_{T-s} f)^{2}(x)\\
		&=\int_{s}^{t} \frac{\d}{\d u}P_{u}(P_{T-u} f)^{2}(x)\, \d u\\
		 &=\int_{s}^{t}\big(-LP_{u}(P_{T-u} f)^{2}(x)+P_{u}[2 P_{T-u} f (L P_{T-u} f)](x)\big)\,\d u\\
		 &=\int_{s}^{t} P_{u}\Big(-L(P_{T-u} f)^{2}+2 P_{T-u} f  L P_{T-u} f\Big)(x)\, \d u\\
		 &= \int_{s}^{t} P_{u} |\nabla P_{T-u} f (x)|^{2} \,\d u\\
		 &= \mathbb{E}_{x}\left(\int_{s}^{t} |\nabla  P_{T-u} f(X_{u})|^{2}\,\d u\right).\end{split}	
\end{equation*}
Thus, $\big\{M_{t}^2-\int_{0}^{t} |\nabla P_{T-s} f\left(X_{s}\right)|^2 \,\d s, \mathcal{F}_t\big\}_{t\ge 0}$ is a martingale, and
\begin{equation*}\label{eq2}\begin{split}
	\langle M\rangle_{t}
	&=\int_{0}^{t} |\nabla P_{T-s} f(X_{s})|^2\, \d s\\
	&=\int_{0}^{t} \sum_{y \in \mathbb{Z}} [P_{T-s} f(X_{s}+y)-P_{T-s} f(X_{s})]^{2} K_s(y)\, \d s .
\end{split}\end{equation*}\end{proof}

Let $(p_t)_{t>0}$ be the heat kernel of $(P_t)_{t>0}$. For every $f\in l^1$ and every $x\in\Z$, set
$$G_{*}(f)(x):=\left(\int_{0}^{\infty} \sum_{y,z\in \mathbb{Z}} |P_{t} f(z+y)-P_{t} f(z)|^{2} p_{t}(x, z)  K_s( y)\, \d t\right)^{1 / 2},$$
and
$$G_{*,T}(f)(x):=\left(\int_{0}^{T} \sum_{y,z\in \mathbb{Z}} |P_{t} f(z+y)-P_{t} f(z)|^{2} p_{t}(x, z)   K_s(y)\, \d t\right)^{1 / 2}.$$
Note that, for every $x\in\Z$, as $T \rightarrow \infty$, we have $G_{*, T}(f)(x)$ increases to $G_{*}(f)(x).$

Indeed, we have following crucial formula for $G_{*, T}(f)$ which,  loosely speaking,  expresses $G_{*, T}(f)$ as the conditional expectation of the predictable quadratic variation of the martingale $M_t$ introduced above.
\begin{lemma}\label{mart-rep}
Let $T>0$. For every $f\in l^1$,
$$
G_{*, T}(f)^2(x)=\sum_{z\in \mathbb{Z}} \mathbb{E}_{z}\big(\langle M\rangle_{T}|X_{T}=x\big) p_{T}(z,x) ,\quad x\in \Z. $$
\end{lemma}
\begin{proof} Indeed, by Lemma \ref{pre-var}, we have
\begin{equation*}\begin{split}
&\sum_{z\in \mathbb{Z}}\mathbb{E}_{z}\big(\langle M\rangle_{T}|X_{T}=x\big) p_{T}(z,x) \\
	&=\sum_{z\in \mathbb{Z}}\mathbb{E}_{z}\left(\int_{0}^{T} \sum_{y\in \mathbb{Z}}\left|P_{T-s} f\left(X_{s}+y\right)-P_{T-s} f\left(X_{s}\right)\right|^{2} K_s(y)\, \d s \bigg|X_{T}=x\right) p_{T}(z,x) \\
	&=\sum_{z\in \mathbb{Z}}\left(\sum_{w\in \mathbb{Z}}\frac{p_{s}(z,w)p_{T-s}(w,x)}{p_{T}(z,x)}\int_{0}^{T}\sum_{y\in \mathbb{Z}}|P_{T-s} f(w+y)-P_{T-s} f(w)|^2K_s(y)\,\d s\right) p_{T}(z,x) \\
	&=\int_{0}^{T} \sum_{w\in \mathbb{Z}} \sum_{y\in \mathbb{Z}}  |P_{s} f(w+y)-P_{s} f(w)|^{2}p_{s}(w,x)K_s(y)\,\d s\\
	&=G_{*,T}(f)^{2}(x),\quad x\in\Z.
\end{split}\end{equation*}
\end{proof}

Now we are ready to complete the proof of  Theorem \ref{main}.
\begin{proof}[Proof of Theorem \ref{main} on $G$]
Let $q\in[2,\infty)$, $T>0$ and $f\in l^1$. Clearly, $\Vert \widetilde{G} \Vert_q$ and $\Vert H \Vert_q$ is bounded by $\Vert G \Vert_q$. Hence, it suffices to prove that $G$ is bounded in $l^q$.

Denote the quadratic variation of $M$ by $[M]$. Then, by \cite[Remark 11.5.8]{cohen} (see also Lemma 6 on page 75 of \cite{LipShi1989} for the general case where $T$ is a stopping time),  there exists some constant $C_q>0$ depending only on $q$ such that
\begin{equation}\label{mart-ineq}
\E\left(\langle M \rangle _T^{\frac{q}{2}}\right) \leq C_q  \E\left([M]_T^{\frac{q}{2}}\right).
\end{equation}
According to Lemma \ref{mart-rep}, Jensen's inequality, \eqref{mart-ineq} and the Burkholder--Davis--Gundy inequality (see e.g. \cite[Theorem 11.5.5]{cohen}), we obtain
\begin{eqnarray*}
	\sum_{x\in \mathbb{Z}} G_{*,T}(f)^q(x)
	&=&\sum_{x\in \mathbb{Z}} \left( \sum_{z\in \mathbb{Z}} \mathbb{E}_{z}(\langle M \rangle_{T}| X_T=x)p_{T}(z,x)\right)^\frac{q}{2}\\
	&\le&\sum_{x\in \mathbb{Z}} \sum_{z\in \mathbb{Z}} \mathbb{E}_{z}\left(\langle M \rangle_{T}^{\frac{q}{2}}| X_T=x\right) p_{T}(z,x)\\
	&=&\sum_{z\in \mathbb{Z}} \mathbb{E}_{z}\left(\langle M \rangle_{T}^{\frac{q}{2}}\right)\\
&\leq& C_q\sum_{z\in \mathbb{Z}} \mathbb{E}_{z}\left( [M]_{T}^{\frac{q}{2}}\right)\\
	&\le& C_{q}'\sum_{z\in \mathbb{Z}} \mathbb{E}_{z}\big(|M_{T}|^q\big) \\
	&\le& C_{q}''\sum_{z\in \mathbb{Z}} \big( \mathbb{E}_{z}|f(X_{T})|^{q}+P_{T}|f|^q(z)\big)\\
	&\leq& 2C_{q}''\sum_{z\in \mathbb{Z}}|f(z)|^q,
\end{eqnarray*}
where in the last two inequalities we applied the elementary inequality, i.e., $(a+b)^q\leq C'''_q(a^q+b^q)$ for every $a,b\geq0$, and the contraction property of $P_T$ in $l^q$, where $C'_q,C''_q, C'''_q$ are positive constants depending only on $q$. Taking $T\rightarrow\infty$, by the monotone convergence theorem, we have
$$\sum_{x\in \mathbb{Z}} G_{*}(f)^q(x) \le 2C_{q}''\sum_{x\in \mathbb{Z}}|f(x)|^q.$$

We claim that $G (f)(x)\le \sqrt{2} G_{*}(f)(x)$, $x\in \mathbb{Z}.$ Indeed,
\begin{equation*}\begin{split}
	G^{2} (f)(x)&=\int_{0}^{\infty}\sum_{y\in \mathbb{Z}}|P_{t} f(x+y)-P_{t} f(x)|^{2}K_s(y)\,\d t\\
	&\le \int_{0}^{\infty}\sum_{y\in \mathbb{Z}} P_{\frac{t}{2}} \big|P_{\frac{t}{2}} f(x+y)-P_{\frac{t}{2}} f(x)\big|^{2} K_s(y)\,\d t\\ &=\int_{0}^{\infty}\sum_{z\in \mathbb{Z}}\sum_{y\in \mathbb{Z}}\big|P_{\frac{t}{2}} f(z+y)-P_{\frac{t}{2}} f(z)\big|^{2} p_{\frac{t}{2}}(x,z) K_s(y)\,\d t \\
	&=2\int_{0}^{\infty}\sum_{z\in \mathbb{Z}} \sum_{y\in \mathbb{Z}} |P_{t} f(z+y)-P_{t} f(z)|^{2} p_{t}(x,z)K_s(y)\,\d t \\
	&=2G_{*}^{2} (f)(x),\quad x\in\Z,
\end{split}\end{equation*} where the inequality is due to that
\begin{eqnarray*}
&&|P_t(x+y)-P_tf(x)|=|P_{t/2}[P_{t/2}f(\cdot+y)](x)-P_{t/2}[P_{t/2}f](x)|\\
&&=|P_{t/2}\{P_{t/2}f(\cdot+y)-P_{t/2}f(\cdot)\}(x)|\leq P_{t/2}|P_{t/2}f(\cdot+y)-P_{t/2}f(\cdot)|(x).
\end{eqnarray*}

Thus, we arrive at
$$\Vert G(f) \Vert_{q}\le c_q\Vert  f \Vert_{q},\quad f\in l^1,$$
for some constant $c_q>0$ depending only on $q$.

Finally, for every $f\in l^q$, by the density and Fatou's lemma, we finish the proof of the boundedness of $G$ in $l^q$ for all $q\in[2,\infty)$.
\end{proof}

\subsection*{Acknowledgment}\hskip\parindent
The authors would like to express their appreciation to the anonymous referee for his/her in-depth comments, suggestions and corrections, which have greatly improved the manuscript. The second named author would like to thank  Prof. Jianhai Bao and Prof. Feng-Yu Wang for helpful comments.

\subsection*{Declarations}
\textbf{Conflict of interest} The authors declare that they have no conflict of interests or personal
relationships that could have appeared to influence the work reported in the present paper.

\end{document}